\documentclass[12pt]{amsart}

\usepackage{hyperref}
\usepackage[margin=1in]{geometry}  % set the margins to 1in on all sides
\usepackage{graphicx}              % to include figures
\usepackage{amsmath}               % great math stuff
\usepackage{amsfonts}              % for blackboard bold, etc
\usepackage{amsthm}                % better theorem environments
\usepackage{enumerate}

\newtheorem{thm}{Theorem}[section]
\newtheorem{lem}[thm]{Lemma}

\begin{document}

\title[Simple closed curves on the punctured Klein bottle]{A Classification of Fundamental Group Elements Representing simple closed curves on the punctured Klein Bottle}
\author{Daniel Gomez }
\email{Danniel.gomez87@gmail.com} 

\maketitle

\begin{abstract}
\bf In this paper we provide a classification of fundamental group elements representing simple closed curves on the punctured Klein bottle, Similar to the Birman-Series classification of curves on the punctured torus\cite{Be}.  In the process, an explicit description of the mapping class group is given.  We then apply this to give a counterexample the simple loop conjecture for representations from the Klein bottle group to $\mathit{PGL}(2,\mathbb{R})$.
\end{abstract}

\section{\textbf{Introduction}}
\begin{figure}[ht]
\begin{center}
\includegraphics[scale=.5]{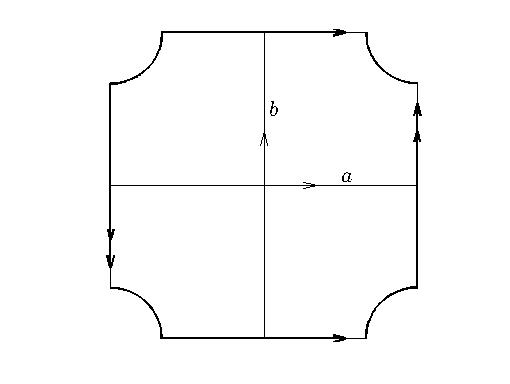}
\caption{A depiction of $K$ where $\pi_1(K)=\left<a,b\right>$}
\end{center}
\end{figure}

The fundamental group of a punctured torus $T$ is generated by a pair of elements represented by simple closed curves intersecting once.  Given such a generating set, Theorem 5.1 of \cite{Be} explicitly describes all words in $\pi_1 T$ that are represented by simple closed curves on $T$.  This has proven useful in various contexts, see eg.~\cite{Na} and \cite{Ma}.  The main result of this paper is Theorem 3.3, which describes the analogous class $\mathcal{C}$ of fundamental group words represented by simple closed curves on the punctured Klein bottle, $K$, in terms of the generators pictured in Figure 1.

The strategy of the proof is as follows.  In Section 2 we record a consequence of work of Korkmaz \cite{Ko}, Lemma 2.2, which gives an explicit three-element generating set for the mapping class group of $K$.  In Lemma 3.2 we list four elements of $\mathcal{C}$ with the property that every simple closed curve on $K$ is mapping class group-equivalent to one representing a list element (this follows from the classification of surfaces).  We compute the action of each generator on each list element, and use this to show that the mapping class group action on $\pi_1 K$ preserves $\mathcal{C}$ with four orbits, one for each of the original elements.

In the final section of this paper,  we provide a counter example to the simple loop conjecture for the punctured Klein Bottle in $PGL(2,\mathbb{R})$. The simple loop conjecture states, any non-injective homomorphism from a closed orientable surface group to another closed orientable surface group must contain an element representing a simple closed curve in the kernel, which was proved true in 1985 \cite{Ga}. In 2014 Katherine Mann shows that on all compact orientable surfaces with boundary and genera at least 1 there are non-injective representations in $SL(2,\mathbb{R})$ which kill no simple closed curve \cite[Theorem 1.2]{Ma}.  She then extends the result to  to nonorientable surfaces of negative Euler characteristic and  nonorientable genus $g\geq 2$, not the punctured Klein bottle nor the closed nonorientable genus-3 surface by changing the target group to $PGL(2,\mathbb{R})$\cite[Theorem 1.3]{Ma}.  Our work uses Theorem 3.3 to extend \cite[Theorem 1.3]{Ma} to the punctured Klein bottle.

 \section{\textbf{Mapping Class Group of $K$} }

To give a description of the mapping class group of the punctured Klein bottle, $K$, we define the necessary diffeomorphisms to generate the mapping class group and provide notation for each. We then prove these particular maps are indeed the elements that generate the group.

The first is a diffeomorphism  called the $Dehn\  twist$. It is supported in an oriented regular neighborhood of any simple closed curve. In the case of the punctured Klein bottle, the necessary Dehn twist is defined about the simple closed curve represented by $b \in \pi_1(K)$, which  can be written explicitly as the map  $S^1\times I$ defined by $(\theta, t)\to(\theta+2\pi t,t)$, since the curve represented by $b$ is simple.  Intuitively, this can be thought of by separating the surface along $b$, rotating one end by $2\pi$ and gluing the surface back together. We borrow notation from Korkmaz and denote the Dehn twist about $b$ and its isotopy class $t_b$\cite{Ko}. Below record the action of $t_b$ on the generators, $a$ and $b$, of the $\pi_1(K)$ group from Figure 1.
\[ t_b: \left\{ \begin{array}{l}
      a \mapsto ab\\
        b \mapsto b\end{array} \right. \]
The $cross\  cap\  slide$, or $Y\!-\!homeomorphism$ is a map defined on non-orientable surface with genus at least two.  Consider the punctured Klein bottle as obtained from an open Mobius band $M$ with one hole (ie. a missing open disk) by identifying antipodal points on the boundary of the hole.  The map is induced by sliding the hole around the core of $M$, then again identifying points of its boundary \cite{Ko}.  Alternatively, If we consider $K$ to be $S^1\times I$ with its ends identified in opposite orientation, with one puncture as above, we can describe The $Y-homomorphism$ as induced by the reflection of the cylinder about $S^2\times \frac{1}{2}$\cite{Li}. We reserve $y$ to represent the $y$-homeomorphism and the associated isotopy class.

\[ y: \left\{ \begin{array}{l}
      a \mapsto a^{-1}\\
       b \mapsto b\end{array} \right. \]

The last of the mapping class group generators is the diffeomorphism supported in a neighborhood of an orientation reversing simple closed curve defined by pushing the puncture once along the curve. This action is called the $boundary\ slide$, and acts trivially along the boundary of any puncture\cite{Ko}. 
     \begin{figure}[ht]
\begin{center}
\includegraphics[scale=.4]{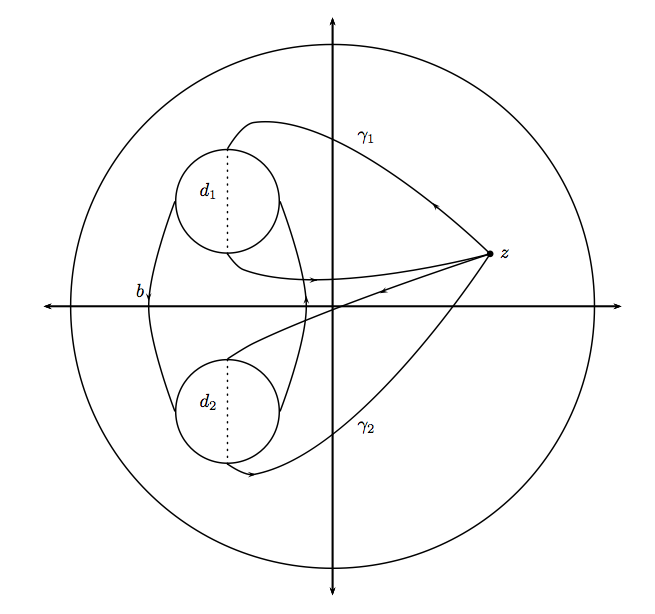} 
 \caption{Another model of $K$ in the plane}
\end{center}

\end{figure}
\begin{figure}[ht]
\begin{center}
 \includegraphics[scale=.4]{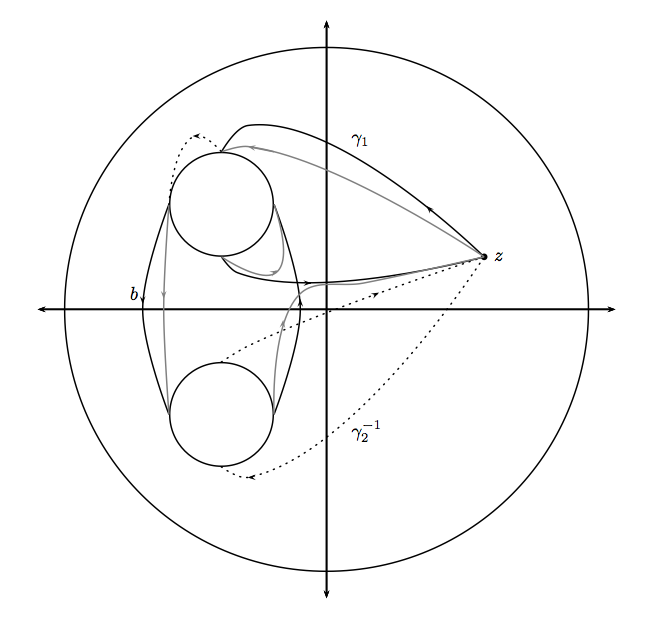}
 \caption{ $\gamma_2^{-1}=t_b(\gamma_1)$ after isotopy.}
\end{center}

\end{figure}

Our description of the mapping class group generators is a corollary to the following result of Korkmaz\cite[Theorem 4.9]{Ko}:

\begin{thm} Let S be a non orientable surface of genus 2 with 1 puncture.  The Pure mapping class group  $\mathcal{M}\mathcal{P}_{2,1}$ of $S$ is generated by $\{t_b, y, w_1,v_1\}$, where $w_1$ and $v_1$ are the boundary slides aroung the curves $\gamma_1$ and $\gamma_2 $ of Firgure 2 respectively.\end{thm}

Here think of the punctured Klein bottle as the punctured 2-sphere after removing 2 interior open discs $d_1$ and $d_2$ then identifying antipodal points on each boundary. The 2-sphere is the compactification of the cartesian plane. Let the puncture be the removal of the point  $z$ . Let $\gamma_1$ be the simple closed curve based at $z$ traveling through $d1$ and  $\gamma_2$ be another simple closed curve based at $z$ traveling through $d2$. In this picture, $a$ is represented by the boundary of $d_2$, and b is a path connecting $d_1$ and $d_2$, Figure 2.   Let $w_1$ denote the boundary slide around $\gamma_1$ and let $v_1$ be the boundary slide around $\gamma_2$.

\begin{lem}
The mapping class group of $K$ is generated by $\{t_b,y,w_1\}$.
\end{lem}
\begin{proof}Since $K$ has only one puncture, the mapping class group, and the pure mapping class group coincide.  In Figure 3 we show $t_b(\gamma_1)=\gamma_2^{-1}$.  It follows that the boundary slide around $\gamma_1^{-1}$ is $t_b(\gamma_1)w_1t_b(\gamma_1)^{-1}=v_1^{-1}$.  Thus there are three generators of the mapping class group $\{t_b,y,w_1\}$. 
\end{proof}

The boundary slide along $\gamma_1$ operates on the generators of $\pi_1(K)$ as follows:

\[ w_1: \left\{ \begin{array}{l}
      a \mapsto a\\
       b \mapsto b^{-1}\end{array} \right. \]

\section{\textbf{Classification of Simple Closed Curves on $K$}}

To classify simple closed curves on $K$, we rely heavily on the classification theorem for compact 2-manifolds, Stated in Massey, \cite[Theorem 5.1]{Mas}:

\begin{thm} Any compact 2-manifold is homeomorphic to either a sphere, connected sum of tori, or a connected sum of projective planes $\mathbb{R}\bf P^2$. 
\end{thm}
We say that a sphere has genus zero, and that a connected sum of tori or projective planes has genus equal to the number of summands. We declare the genus of a surface with boundary to be that of the closed surface obtained by joining a disk to each boundary component.  It should also be noted that genus is a topological invariant.

This result allows to distinguish possible simple closed curves $\gamma$ by Identifying the resultant surfaces after cutting along the specified curve.

\begin{lem} Every simple closed curve, $\gamma$, on $K$ is mapping class group equivalent to either $a, a^2, b,$ or $ab^{-1}a^{-1}b^{-1}$.\end{lem} \bigskip

 It is useful to recall the Euler characteristic of an orientable surface is given by $\chi(S)=2-2g-i$ or $\chi(S)=2-g-i$ in the non-orientable case, and that the Euler characteristic is invariable under the removal of any simple closed curve,  $\chi(S)$=$\chi(S-\gamma)$.

\begin{proof}
 We consider the following cases:
\ \ \begin{enumerate}[(i)]
\item  $\gamma$  separates $K$ into two orientable surfaces $S_1$ and $S_2 $.

This case is clearly impossible since a union of two orientable surfaces must be orientable.

\item $\gamma$ separates  $K$ into two non-orientable surfaces  $S_1$ and  $S_2$.

We let $S_i$ have $i$ boundary components which restricts the genera to the relation $g_1+g_2=2$, where $g_i$ is a natural number. As such, $g_1=g_2=1$ is the only possibility. This implies $S_1$ is a Mobius band and $S_2$ is the Mobius band with one puncture.  This is produced by cutting $K$ along the closed curve represented by $a^2$ in the fundamental group.

\item $\gamma$ separates $K$  into an orientable surface $S_1$  and  a non-orientable  surface  $S_2.$
\subitem(a) $S_1$ is punctured.

Again the relation $\chi_1+\chi_2=-1$ restricts the genera to satisfy $2g_1+g_2=2$ with $g_1 \geq 0$ and $g_2 \geq 1$.  When $g_1=0$, $g_2=2$, which implies $S_1$ is the surface of Euler characteristic zero with two boundary components, i.e. the Annulus.  $S_2$ is the surface of genus $2$ with  Euler characteristic $-1$ and a single boundary component., i.e. the punctured Klein bottle. The disjoint union of the annulus and punctured Klein bottle arise by cutting  $K$ along $ab^{-1}a^{-1}b^{-1}$ in our representation of $\pi_1(K)$.
 
\subitem(b) $S_2$ is punctured.

 When the puncture is on the non-orientable surface, the same relation on the genera holds. When $g_1=0$, $S_1$ is a disk, so as in case (i), is not considered.

\item  $\gamma$ is non-separating and the resulting surface is orientable.
\subitem(a) $S$  has 2 boundary components.

The Euler characteristic is given by $\chi=2-2g-2=-1$.  There is not natural number satisfying the equation, so this case is not possible.

\subitem(b) $S$ has 3 boundary components.

$\gamma$ cuts $ K$ such that the resulting surface has $3$ boundary components, and satisfies $\chi=2-2g-3=-1$ which forces $g$ to be zero, and the surface to be the 3-holed sphere, attained by cutting along the generator $b$ of $\pi_1(K)$.

\item  $\gamma$ is non-separating and the resulting surface is non-orientable.

\subitem(a) $S$  has 2 boundary components.

$\gamma$ cuts $ K$ such that the resulting surface has $2$ boundary components, and satisfies $\chi=2-g-2=-1$, which forces $g$ to be $1$, so the surface to be the once punctured Mobius band and is produced by  Cutting $K$ along the curve represented by $a\in\pi_1(K)$.

\subitem(b) $S$ has 3 boundary components.

The Euler characteristic is $\chi=2-g-2=-1$, which implies $g=0$. The genus of a non-orientable surface must be greater than or equal to 1, so this case is disqualified.

\end{enumerate}
\end{proof}
\begin{thm} Let $\mathcal{C}$ be the set of words of the following following form, and their inverses:
\begin{enumerate}[(i)]

\item  $a^{\pm1}$, $a^{\pm2}$,  or $b^{\pm1}$
\item $ab^{-1}a^{-1}b^{-1}$
\item  $ab^n$ or $ab^nab^n$ for $n\in\mathbb{Z}$

\end{enumerate}

Every simple closed curve on $K$ represents a word in $\mathcal{C}$, and every word in $\mathcal{C}$ is represented by a simple closed curve. 

\end{thm}
This comports with the description of the punctured Klein bottle's complex of curves given on \cite[pp.~175--176]{Scharlemann}, with $\bar{a}$ there in the role of $b$ here and vice-versa.  (Note that the boundary is ignored in the complex of curves of \cite{Scharlemann}, as are the squares of one-sided curves.)\

We use the following lemmas to prove Theorem 3.3.
 \begin{lem}
 The  mapping class group preserves the structure of $\mathcal{C}$, i.e. for any element $[\phi]$ in the mapping class group, and any  word form $c\in\mathcal{C}$, $\phi(c)$ is conjugate into $\mathcal{C}$. 
 \end{lem}
 \begin{proof}\ \ \ 

\textit{\underline{Case 1:}} $t_b(c)\in\mathcal{C}$ for any $c\in\mathcal{C}$.

By definition of the Dehn twist around $b$, $t_b(a)=ab$, $t_b(a^2)=abab$, and $tb(b)=b$.  

On the boundary curves beginning with $a$, $t_b$ acts trivially;  $t_b(ab^{-1}a^{-1}b^{-1})=ab^{-1}a^{-1}b^{-1}$.

 Applying $t_b$ to $a^{-1}b^n$ gives $b^{-1}a^{-1}b^n$, which is conjugate to $a^{-1}b^{n-1}$, and  $t_b$ applied to $ a^{-1}b^{n}a^{-1}b^{n}$ equals $b^{-1}a^{-1}b^{n-1}a^{-1}b^{n}$ which is conjugate to $a^{-1}b^{n-1}a^{-1}b^{n-1}$. Finally, $t_b$ maps the elements $ab^{n}$ and $ab^{n}ab^{n}$ to $ab^{n+1}$ and $ab^{n+1}ab^{n+1}$ respectively.  \bigskip

\textit{\underline{Case 2:}} $w_1(c)\in\mathcal{C}$ for any $c\in\mathcal{C}$.

Our boundary slide fixes $a$ and takes  $b$ to its inverse, so acting on any word in $\mathcal{C}$ will

 preserve the structure of the product exchanging powers of $b$ for their negative.\bigskip

 \textit{\underline{Case 3:}} $y(c)\in\mathcal{C}$ for any $c\in\mathcal{C}$.
 \nopagebreak
 Similarly, the $Y$-homeomorphism  fixes $b$ and takes $a$ 
 
 to its inverse, so the argument in the previous case applies, which concludes this Lemma.
 \end{proof}
  \begin{lem}
 Let $\mathcal{C}_0=\{a,b,a^{2},ab^{-1}a^{-1}b^{-1}\}$ for any $c\in\mathcal{C}$ there is a mapping class group element  $[\phi]$ such that $\phi(c)$ is conjugate to some $c_0\in\mathcal{ C}_0$ or its inverse.
  \end{lem}
 
 \begin{proof}  Clearly $y:a\mapsto a^{-1}$ , $y:a^{-2}\mapsto a^{2}$,  and $w_1:b\mapsto b^{-1}$, and it is obvious the maps hold in the other direction as well.

The boundary slide takes $ab^{-1}a^{-1}b^{-1}$ to $ aba^{-1}b$ and again can be applied in the opposite direction. An analogous statement is  true for  $w_1$ acting on $a^{-1}b^{-1}ab^{-1}$ and  $ a^{-1}bab$.

The Dehn twist around $b$  maps $a\mapsto ab$ and by induction, $t_b^{n}(a)=ab^{n}$.  and similarly, $t_b: a^{2}\mapsto abab$ and again, $t_b^{n}(a^{2})=ab^{n}ab^{n}$ follows by induction.
Applying $y$ to $ab^{n}ab^{n}$ and $ab^{n}$ yields $a^{-1}b^{n}a^{-1}b^{n}$ and $a^{-1}b^{n}$ respectively.  Acting on  $ab^{n}ab^{n}$, $ab^{n}$, $a^{-1}b^{n}a^{-1}b^{n}$ and $a^{-1}b^{n}$ with the boundary slide produces a representation of the remaining classes of words in $\mathcal{C}$.
\end{proof}

Now, we prove Theorem 3.3.
\ \ \begin{proof}
Let $c$ be a simple closed curve on $K$.  Lemma 3.2 implies there is a mapping class group element such that $\phi(c)$ is conjugate to some $c_0\in\mathcal{C}_0$.  It follows by Lemma 3.4 that $\phi^{-1}(c_0)$ which represents $c$ is conjugate to a word in $\mathcal{C}$. Since every word in $\mathcal{C}_0$ represents a simple closed curve, and mapping classes take simple closed curves to simple closed curves, Lemma 3.5 impies every word in $\mathcal{C}$ represents a simple closed curve.
\end{proof}

 Thus, we have an explicit description of Simple closed curves on $K$ up to inner automorphism, which corresponds to words permitted in $\mathcal{C}$.

\section{\textbf{ Counter example to the Simple Loop Conjecture for $PGL(2,\mathbb{R})$}}

In this section we rely on a representation of $\pi_1(K)$ in $PGL(2,\mathbb{R})$.  Since our simple closed curves represented by  $a$ and $b$ freely generate  $\pi_1(K)$, any choice of $\alpha$ and $\beta$ such that $\alpha,\beta\neq0$ defines a representation $\rho:\pi_1(K)\to PGL(2,\mathbb{R})$.

\[ \rho(a)=\left( \begin{array}{cc}\alpha & 0 \\   0 & \alpha^{-1}  \end{array} \right), \ \ \ \ \  \ \ \ \  \rho(b)=\left( \begin{array}{cc}\beta & 1 \\  0 & \beta^{-1}  \end{array} \right)    \]    

This representation allows for the extension of \cite[Theorem 1.3]{ Ma} to the punctured Klein bottle as follows.
\begin{thm} Let $K$ be the punctured  Klein bottle;  Let $\alpha$ and $\beta$  satisfy $\alpha^{k}\beta^{l}\neq\pm1$ for any integers $k,l\in\mathbb{Z}\!\setminus\!\{0\}$ (note that it is sufficient to choose $\alpha$ and $\beta$ to be relatively prime integers). Then $\rho:\pi_1(K)\to PGL(2,\mathbb R)$ satisfies

\begin{enumerate}[1.]
\item $\rho$ is not injective.

\item If $\rho(\alpha)=\pm I$, then $\alpha$ is not represented by a simple closed curve.

\item If $\alpha$ is represented by a simple closed curve, then $\rho(\alpha^k)\neq I$ for any $k\in \mathbb Z \!\not\{0\}$

\end{enumerate}
\end{thm}

\begin{proof}
To show The representation is non-injective, we first show $\pi_1(K)$ is not solvable since it has a quotient  which is not solvable. This fact follows since there exists a homomorphism $f:\pi_1(K)\to A_5$ where $A_5=\left<(123),(345)\right>$ the alternating group of even permutations on $5$ elements defined by $a\mapsto(123)$ and $b\mapsto(345)$\cite[Pg. 111, Exercise 16]{Du}.  The group $A_5$ is non commutative and simple. If  we assume for a contradiction that $A_5$ is solvable, by simplicity,  the only composition series is $A_5\trianglerighteq1$, which implies the quotient group $A_5/1=A_5$ is abelian since quotient groups of consecutive terms of a composition series are abelian. This is our contradiction, thus $A_5$ is not solvable.  It is a common fact that a quotient of a solvable group is solvable, which by contrapositive, shows $\pi_1(K)$ is not solvable.

Our representation $\rho$ of $\pi_1(K)$ in contrast is $2$ step solvable since any element $X$ in $G^{1}=[PGL(2,\mathbb{R}),PGL(2,\mathbb{R})]$ is of the form,
 \[ X=\left( \begin{array}{cc}1& x \\   0 & 1  \end{array} \right), \ \ \ \ \  \ \ \ \  X^{-1}=\left( \begin{array}{cc}1 & -x\\  0 & 1  \end{array} \right)    \] . And by direct computation, we see that any element in $G^{2}=[G^{1},G^{1}]$ is given by a product \[XYX^{-1}Y^{-1}= \left( \begin{array}{cc}1& x+y-x-y \\   0 & 1  \end{array} \right)= \left( \begin{array}{cc}1 & 0\\  0 & 1  \end{array} \right)  \]Thus, $\rho$ has a composition series terminating in the trivial group. This implies $\rho$ is solvable and non-injective proving $1$ from Theorem 4.1.
 
To show  consequences $2$ and $3$ of Theorem 4.1, it is enough to note that products of upper triangular matrices are upper triangular, and as a consequence have $1,1$-entries in the form of words in $\mathcal{C}$ prior to reduction in $\mathbb{R}$.   Since no word form in $\mathcal{C}$ reduces to a nontrivial product of $\alpha$ and $\beta$ or their inverses in $\mathbb{R}$, the top left entry of any matrix representing a simple closed curve in PGL$(2,\mathbb{R})$ is never of the form $\alpha^{k}\beta^{l}$ with $k$ and $l$ both $0$.  Thus, no matrix representation or any power of a simple closed curve can be the identity.  No power of simple closed curve is in the kernel of our representation $\rho$.
\end{proof}

\bibliographystyle{plain}
\bibliography{thesisbib}

\begin{thebibliography}{1}

\bibitem{Be}
Joan~S Birman and Caroline Series.
\newblock An algorithm for simple curves on surfaces.
\newblock {\em J. London Math. Soc.(2)}, 29(2):331--342, 1984.

\bibitem{Du}
David~Steven Dummit and Richard~M Foote.
\newblock {\em Abstract algebra}, volume~3.
\newblock Wiley Hoboken, 2004.

\bibitem{Ga}
David Gabai et~al.
\newblock The simple loop conjecture.
\newblock {\em Journal of Differential Geometry}, 21(1):143--149, 1985.

\bibitem{Ko}
Mustafa Korkmaz.
\newblock Mapping class groups of nonorientable surfaces.
\newblock {\em Geom. Dedicata}, 89:109--133, 2002.

\bibitem{Li}
W.~B.~R. Lickorish.
\newblock Homeomorphisms of non-orientable two-manifolds.
\newblock {\em Proc. Cambridge Philos. Soc.}, 59:307--317, 1963.

\bibitem{Ma}
Kathryn Mann.
\newblock A counterexample to the simple loop conjecture for {$\rm {PSL}(2,\Bbb
  R)$}.
\newblock {\em Pacific J. Math.}, 269(2):425--432, 2014.

\bibitem{Mas}
William~S. Massey.
\newblock {\em A basic course in algebraic topology}, volume 127 of {\em
  Graduate Texts in Mathematics}.
\newblock Springer-Verlag, New York, 1991.

\bibitem{Na}
Toshihiro Nakanishi.
\newblock A series associated to generating pairs of a once punctured torus
  group and a proof of {M}c{S}hane's identity.
\newblock {\em Hiroshima Math. J.}, 41(1):11--25, 2011.

\bibitem{Scharlemann}
Martin Scharlemann.
\newblock The complex of curves on nonorientable surfaces.
\newblock {\em J. London Math. Soc. (2)}, 25(1):171--184, 1982.

\end{thebibliography}

\end{document}